\documentclass[preprint,11pt]{elsarticle}
\usepackage{lineno,hyperref}
\modulolinenumbers[5]

\usepackage{color}
\usepackage{amsmath}
\usepackage{amsfonts,amsthm,amssymb}
\usepackage{amsthm}
\usepackage{amsfonts}
\usepackage{graphics}
\usepackage{pstricks}
\usepackage{graphicx}
\usepackage{cleveref}
\usepackage{extarrows,chngpage,array,float}
\usepackage{amssymb}
\usepackage{latexsym,bm}
\usepackage{tikz-cd}
\usepackage{natbib}
\usepackage{algorithm}
\usepackage{algorithmic}
\newtheorem{theorem}{Theorem}[section]
\newtheorem{lemma}[theorem]{Lemma}
\newtheorem{remark}[theorem]{Remark}
\newtheorem{definition}[theorem]{Definition}
\newtheorem{proposition}[theorem]{Proposition}
\newtheorem{corollary}[theorem]{Corollary}

\newtheorem{problem}[theorem]{Problem}

\newtheorem{example}[theorem]{Example}
\allowdisplaybreaks 

\journal{xxx}

\begin{document}
	
	\begin{frontmatter}
		
		\title{On the maximum twist width of delta-matroids}

		\author{Xian'an Jin$^{1,2}$,\ \ Zhuo Li$^{1}$\footnote{Corresponding author.},\ \ Qi Yan$^{3}$,\ \ Gang Zhang$^{1,4}$\\
        ~\\
        \small $^{1}$School of Mathematical Sciences, Xiamen University, P. R. China\\
        \small $^{2}$School of Mathematics and Statistics, Qinghai Minzu University, P. R. China\\
		\small $^{3}$School of Mathematics and Statistics, Lanzhou University, P. R. China\\
        \small $^{4}$School of Mathematics and Statistics, Minnan Normal University, P. R. China\\
               				\small\tt Email: xajin@xmu.edu.cn, lzhuo@stu.xmu.edu.cn, yanq@lzu.edu.cn, gzh\_ang@163.com}
		
		\begin{abstract}
        For a ribbon graph  $G$, let  $\gamma(G)$ denote its Euler genus.  Recently, Chen, Gross and Tucker [J. Algebraic Combin. 63 (2026) 13] derived a formula for the maximum partial-dual Euler-genus $\partial\gamma_M(G)$ of a ribbon graph $G$. Their key finding is that $\partial\gamma_M(G)$ can be achieved by a partial dual with respect to the edge set of a spanning quasi-tree. Moreover, they proposed the following problem: ``~Given a ribbon graph $G$, is there a sequence of edges $e_1,e_2,\dots, e_k$ such that $\gamma(G^{\{e_1, e_2,\dots, e_k\}})=\partial\gamma_M(G)$ and such that the sequence  
			$$\gamma(G), \gamma(G^{\{e_1\}}), \dots, \gamma(G^ {\{e_1, e_2,\dots, e_k\}})$$
			rises monotonically (i.e., never decreasing) to $\partial\gamma_M(G)$?~'' 
            
            Delta-matroids are set systems that satisfy the symmetric exchange axiom and serve as a matroidal abstraction of ribbon graphs. In this paper, we first show that the maximum twist width of a set system can be attained by twisting one of its feasible sets, which extends the result of Chen, Gross and Tucker to set systems. Then we solve the delta-matroid version of their problem, thereby providing an affirmative answer to the original problem for ribbon graphs.
		\end{abstract}
		
		
	\end{frontmatter}
	\section{Introduction}
	The partial dual $G^A$ of a ribbon graph $G$ with respect to an edge subset $A$ was introduced by Chmutov \cite{Chmutov2019}, serving as a generalization of the geometric duality of $G$. There is extensive literature on the partial duals of ribbon graphs. Gross et al. \cite{Gross2020} proposed a method for expressing the Euler genus of a ribbon graph. Chen et al. \cite{Chen2025+} derive a formula for the maximum partial-dual Euler-genus of any ribbon graph. For other related literature, see the references \cite{Chen2025, Huggett2013, Moffatt2012, Metsidik2018}.
	
	Let $k(G)$ denote \emph{the number of connected components} of $G$, and let $v(G), e(G)$, $f(G)$ and $\gamma(G)$ denote the numbers of \emph{vertices, edges, faces} and \emph{Euler genus} of $G$, respectively. By Euler's formula, we have
	$$\gamma(G)=2k(G)-v(G)+e(G)-f(G).$$
	
	A \emph{set system} is a pair $D = (E, \mathcal{F})$, where $E$ (also denoted $E(D)$) is a finite set called the \emph{ground set}, and $\mathcal{F}$ (also denoted $\mathcal{F}(D)$) is a collection of subsets of $E$ called \emph{feasible sets}.  Bouchet \cite{Bouchet1987} introduced \emph{delta-matroids} as a set system $D = (E, \mathcal{F})$ with $\mathcal{F}\neq \emptyset$ satisfying the symmetric exchange axiom. Delta-matroids are matroidal abstractions of ribbon graphs. The relationship between embedded graphs and delta-matroids parallels the relationship between graphs and matroids; much research on delta-matroids originates from ribbon graphs. For additional related work, we refer the reader to \cite{Bouchet1989, Chun2019F, Moffatt2017, Yan2022,  Yan2024, Yuschak2024}.

    For any ribbon graph $G$, there exists a delta-matroid $D(G)$ \cite{Chun2019} naturally associated with it. The twist $D(G) \Delta A$ (an operation defined by Bouchet \cite{Bouchet1987} for delta-matroids) corresponds precisely to taking the partial dual $G^A$ with respect to an edge subset $A$; that is,
\[
D(G) \Delta A = D(G^A).
\]

     Let $D=(E, \mathcal{F})$ be a set system. We denote by $\mathcal{F}_{\max}(D)$ and $\mathcal{F}_{\min}(D)$ the collections of feasible sets of maximum and minimum cardinality, respectively. Let $r_{\max}(D)$ and $r_{\min}(D)$ denote the cardinalities of a maximum and a minimum feasible set, respectively.  The  \emph{width} of $D$, denote by $\omega(D)$, is defined by $$\omega(D):=r_{\max}(D)-r_{\min}(D).$$
      Let $G$ be a ribbon graph and $D(G)$ its delta-matroid. Chun et al. showed in \cite{Chun2019} that the width of $D(G)$ equals to the Euler genus of $G$; that is, $$\gamma(G)=\omega\left(D(G)\right).$$

     Chen et al. \cite{Chen2025+} defined the \emph{maximum partial-dual Euler-genus $\partial\gamma_M(G)$} of a ribbon graph $G$ as
\[
\partial\gamma_M(G) = \max\{\, \gamma(G^A) \mid A \subseteq E(G) \,\}.
\]
Similarly, we define the \emph{maximum twist width $\partial\omega_M(D)$} of a set system $D = (E, \mathcal{F})$ as
\[
\partial\omega_M(D) = \max\{\, \omega(D \Delta A) \mid A \subseteq E \,\}.
\]
In \cite{Chen2025+}, Chen et al. also provided a formula for $\partial\gamma_M(G)$. A key insight of their work is that this maximum is attained by a partial dual with respect to the edge set of a spanning quasi-tree. We extend this insight to set systems by showing that the maximum twist width $\partial\omega_M(D)$ can be attained by twisting some feasible set.
	
	In addition, Chen, Gross, and Tucker proposed the following problem:
	\begin{problem}[\cite{Chen2025+}]\label{oques}
		\normalfont
		For any ribbon graph $G$, is there a sequence of edges $e_1,e_2,\dots, e_k$ such that
		$\gamma(G^ {\{e_1, e_2,\dots, e_k\}})=\partial\gamma_M(G)$ and such that the sequence 
		$$\gamma(G), \gamma(G^{\{e_1\}}), \dots, \gamma(G^ {\{e_1, e_2,\dots, e_k\}})$$
		rises monotonically (i.e., never decreasing) to $\partial\gamma_M(G)$?
	\end{problem}
    
	We extend this problem to the framework of delta-matroids.
    
	\begin{problem}\label{ques}
		\normalfont
		For any delta-matroid $D=(E,\mathcal{F})$, is there a sequence of elements $ e_1, e_2,\dots, e_k$ such that
		$$\omega(D\Delta \{e_1, e_2,\dots, e_k\})=\partial\omega_M(D)$$ and such that the sequence 
		$$\omega(D), \omega(D\Delta \{e_1\}), \dots, \omega(D\Delta \{e_1, e_2,\dots, e_k\})$$
		rises monotonically (i.e., never decreasing) to $\partial\omega_M(D)$?
	\end{problem}

    We provide an affirmative answer to Problem~\ref{ques} (see Theorem~\ref{key}), thereby also resolving Problem~\ref{oques}. In fact, our result further demonstrates that the set $\{e_1, e_2, \dots, e_k\}$ itself can be chosen to be a feasible set of $D$.

    The remainder of this paper is structured as follows. Section~\ref{2} recalls the basic definitions of ribbon graphs, delta-matroids, partial duals, and twists. In Section~\ref{3}, we present our first main result, which extends the result of Chen et al. \cite{Chen2025+} from ribbon graphs to set systems. Section~\ref{4} provides a complete solution to Problem~\ref{ques}. Finally, Section~\ref{5} presents an algorithm for constructing a non-decreasing sequence that attains $\partial\omega_M(D)$ and provides some examples.
	
	\section{Preliminaries}\label{2}
    
	In this section, we recall several definitions related to ribbon graphs and delta-matroids. We refer the reader to \cite{Bouchet1987}, \cite{Chmutov2019}  and \cite{Chun2019} for further details.
    
	\subsection{Ribbon Graphs}
	A ribbon graph arises naturally from a cellularly embedded graph by taking a small neighborhood for each vertex and edge, and is formally defined as follows:
	
	\begin{definition}[\cite{Bollobas2002}]
		\normalfont
		A \emph{ribbon graph} \( G = (V(G), E(G)) \) is a surface with boundary, represented as the union of two sets of topological discs: a set \( V(G) \) of vertices and a set \( E(G) \) of edges, satisfying the following properties:
		\begin{enumerate}
			\item  The vertices and edges intersect in disjoint line segments.
			\item   Each such line segment lies on the boundary of exactly one vertex and exactly one edge.
			\item   Every edge contains exactly two such line segments.
		\end{enumerate}
	\end{definition}

Let \( G \) be a ribbon graph. A \emph{ribbon subgraph} of \( G \) is obtained by deleting some vertices and edges from \( G \). A \emph{spanning ribbon subgraph} is obtained by deleting edges only. A \emph{quasi-tree} \( Q \) is a connected ribbon graph with exactly one boundary component. For a connected ribbon graph \( G \), a \emph{spanning quasi-tree} \( Q \) of \( G \) is a spanning ribbon subgraph with exactly one boundary component. For a disconnected ribbon graph \( G \), we say that a ribbon graph \( Q \) is a \emph{spanning quasi-tree} of \( G \) if \( k(Q) = k(G) \) and if each connected component of \( Q \) is a spanning quasi-tree of the corresponding connected component of \( G \). 

\begin{definition} [\cite{Chmutov2019}]
\normalfont
For a ribbon graph $G$ and $A\subseteq E(G)$, the \emph{partial dual} $G^A$ of $G$ with respect to $A$ is a ribbon graph obtained from $G$ by gluing a disc to $G$ along each boundary component of the spanning ribbon subgraph $(V(G), A)$ (such discs will be the vertex-discs of $G^A$), removing the interiors of all the vertex-discs of $G$ and keeping its edge-ribbons unchanged. 
\end{definition}

In particular, the \emph{geometric dual} of $G$, denoted $G^*$, is the partial dual with respect to $E(G)$, i.e., $G^* = G^{E(G)}$.
	
\subsection{Delta-matroids}

A set system $D$ is \emph{proper} if $\mathcal{F}\neq \emptyset$,  is \emph{trivial} if $E=\emptyset$. Bouchet \cite{Bouchet1987} introduced delta-matroids as follows.
	
	\begin{definition}[\cite{Bouchet1987}]
		\normalfont
		A \emph{delta-matroid} is a proper set system \( D = (E, \mathcal{F}) \) which satisfies the \emph{Symmetric Exchange Axiom}:
		for any \( X, Y \in \mathcal{F} \) and any \( u \in X \Delta Y \), there exists a \( v \in X \Delta Y \) (possibly \( v = u \)) such that \( X \Delta \{u, v\} \in \mathcal{F} \). Here \( X \Delta Y = (X \cup Y) \setminus (X \cap Y) \) denotes the symmetric difference of the sets $X$ and $Y$.
	\end{definition}

    Let $D=(E, \mathcal{F})$ be a delta-matroid. If all feasible sets in $\mathcal{F}$ have the same size, then $D$ is a \emph{matroid}, and the feasible sets are called its \emph{bases}.
	The twist is one of the most fundamental operations in delta-matroid, introduced by Bouchet in \cite{Bouchet1987}.
	\begin{definition}[\cite{Bouchet1987}]
		\normalfont
        Let \( D = (E, \mathcal{F}) \) be a set system and let \( A \subseteq E \). The \emph{twist} of \( D \) with respect to \( A \), denoted by \( D \Delta A \), is the set system \( (E, \mathcal{F} \Delta A) \), where
    \[
    \mathcal{F} \Delta A := \{ X \Delta A : X \in \mathcal{F} \}.
    \]
	\end{definition}

    Let $G = (V(G), E(G))$ be a ribbon graph, and define
\[
\mathcal{F}(G) := \{ F \subseteq E(G) : F \text{ is the edge set of a spanning quasi-tree of } G \}.
\]
The \emph{delta-matroid} of $G$, denoted by $D(G)$, is defined as $D(G) := (E(G), \mathcal{F}(G))$ \cite{Chun2019}.

	Chun et al. \cite{Chun2019} established the following relationship between the partial duals of a ribbon graph $G$ and the twists of its associated delta-matroid $D(G)$.
	\begin{theorem}[\cite{Chun2019}]\label{trans}
    
		Let $G$ be a ribbon graph and $A\subseteq E(G)$. Then \[
    D(G^A) = D(G) \Delta A.
    \]
	\end{theorem}
    
	\section{Maximum  Twist Width}\label{3}

	\begin{theorem}\label{feasible}
		
For any set system \( D = (E, \mathcal{F}) \), there exists a feasible set \( F \in \mathcal{F} \) such that
    \[
    \omega(D \Delta F) = \partial\omega_M(D).
    \]
	\end{theorem}

\begin{proof}
It suffices to show that for any subset $A \subseteq E$, there exists a feasible set $F \in \mathcal{F}$ such that $\omega(D \Delta F) \geq \omega(D \Delta A)$.

 If \( A \in \mathcal{F} \), we may take \( F = A \) and the inequality holds trivially. Assume now that \( A \notin \mathcal{F} \). Choose a feasible set \( F \in \mathcal{F} \) that minimizes \( |A \Delta F| \) (such an \( F \) exists because \( D \) is proper). Note that \( A \Delta F \neq \emptyset \). For each $f\in F\Delta A$, we have $$r_{\min}(D\Delta\{A\Delta f\})=r_{\min}\left((D\Delta A)\Delta f\right)=|\{F\Delta A\}\Delta f|=|F\Delta A|-1=r_{\min}(D\Delta A)-1$$
		and 
		$$r_{\max}(D\Delta\{A\Delta f\})\geq r_{\max}(D\Delta A)-1.$$
		Thus $\omega(D\Delta \{A\Delta f\})\geq \omega(D\Delta A)$.
		Repeating this process continuously for all the elements in $F\Delta A$ yields
		$$\omega(D\Delta F)=\omega(D\Delta A\Delta\{A\Delta F\}) \geq \omega(D\Delta A).$$
        Thus the conclusion holds.
        \end{proof}
		
	For the delta-matroid $D(G)$ of a ribbon graph $G$, each feasible set corresponds to a spanning quasi-tree, and we have $\gamma(G) = \omega(D(G))$. Specializing Theorem~\ref{feasible} to ribbon graphs thus yields the following corollary, which can also be obtained directly from the proof of the main theorem in \cite{Chen2025+}.
    
	\begin{corollary}[\cite{Chen2025+}]
		
		For any ribbon graph $G$, there exists a spanning quasi-tree $A$ of $G$ such that $$\gamma(G^A)=\partial\gamma_M(G).$$
	\end{corollary}
	\begin{proof}
    Let $D(G)$ be the delta-matroid of $G$. By Theorem~\ref{trans}, for any $X \subseteq E(G)$,
    \[
    \gamma(G^X) = \omega(D(G) \Delta X).
    \]
    Hence, $\partial\gamma_M(G) = \partial\omega_M(D(G))$. 
    By Theorem~\ref{feasible}, there exists a feasible set $A$ of $D(G)$ (corresponding to a spanning quasi-tree of $G$) such that 
    \[
    \omega(D(G) \Delta A) = \partial\omega_M(D(G)).
    \]
    Combining these equalities yields 
    $$\gamma(G^A)=\omega(D(G)\Delta A)= \partial\omega_M(D(G))=\partial\gamma_M(G).$$ 
	\end{proof}

    Consider a connected ribbon graph \( G \) and an edge subset \( A \subseteq E(G) \). 
A formula for the Euler genus of the partial dual \( G^A \), given in \cite{Gross2020}, is
\[
\gamma(G^A) = 2 + e(G) - f(A) - f(A^c).
\]
When \( A \) is the edge set of a spanning quasi-tree of \( G \), we have \( f(A) = 1 \), hence
\[
\gamma(G^A) = 1 + e(G) - f(A^c).
\]
Therefore, \( \gamma(G^A) \) attains its maximum if and only if \( f(A^c) \) attains its minimum.
    
	\begin{proposition}[\cite{Chen2025+}]\label{chen}
        Let $G$ be a connected ribbon graph, and let $\mathcal{T}$ denote the set of spanning quasi-trees of $G$. Then
    \[
    \partial\gamma_M(G) = 1 + e(G) - \min_{A \in \mathcal{T}} f(A^c).
    \]
	\end{proposition}
    
	The invariant $\min_{A\in \mathcal{T}}f(A^c)$ is called \emph{partial-duality deficiency} of the ribbon graph $G$, denoted by $\partial\xi(G)$.
	For a set system $D=(E,\mathcal{F})$, Theorem~\ref{feasible} yields
	\begin{align}\label{eqation1}
	    \partial\omega_M(D)=\max\{|F_1\Delta F_2|\mid F_1,F_2\in \mathcal{F}\}.
	\end{align}
    Combining this formula with Proposition~\ref{chen}, we obtain an expression for $\partial\xi(G)$ that depends only on the set of spanning quasi-trees of $G$.
    
	\begin{proposition}
		
		Let $G$ be a connected ribbon graph and $D(G)=(E(G),\mathcal{F}(G))$ be its delta-matroid. Then 
        \[
    \partial\xi(G)=e(G)+1-\max\{|F_1\Delta F_2|\mid F_1,F_2\in \mathcal{F}(G)\}.
    \]
	\end{proposition}
	
	\section{Monotonicity with Respect to Maximum Twist Width}\label{4}

    In this section, we prove the existence of a sequence of elements in a delta-matroid such that the twist width increases monotonically to its maximum. We first present a useful lemma.
    
	\begin{lemma}[\cite{Bonin}]\label{contain}
    
		For any delta-matroid $D=(E,\mathcal{F})$ and any $F_0\in \mathcal{F}$, there exist $F_1\in \mathcal{F}_{\min}(D)$ and $F_2\in \mathcal{F}_{\max}(D)$ such that $F_1\subseteq F_0\subseteq F_2$.
	\end{lemma}
	\begin{theorem}\label{key}
		
		For any delta-matroid $D=(E,\mathcal{F})$, there exists a sequence of elements $e_1, e_2,\dots, e_k$ of $E$ such that
      		\begin{enumerate}
			\item  $\{ e_1, e_2,\dots, e_k\}\in \mathcal{F}$;
			\item   $\omega(D\Delta \{e_1, e_2,\dots, e_k\})=\partial\omega_M(D)$;
			\item  the sequence  of twist widths
		$$\omega(D), \omega(D\Delta e_1), \dots, \omega(D\Delta \{e_1, e_2,\dots, e_k\})$$
       rises monotonically (i.e., never decreasing) to \( \partial\omega_M(D) \).
        
		\end{enumerate}
	\end{theorem}
	\begin{proof}
    Let $\hat{\mathcal{F}}$ be the set of feasible sets of $D$ such that $\omega(D\Delta A)=\partial\omega_M(D)$ holds for every $A\in \hat{\mathcal{F}}$. By Theorem~\ref{feasible}, $\hat{\mathcal{F}}$ is nonempty.
Let $F$ be a feasible set of minimum cardinality in $\hat{\mathcal{F}}$. By this choice, $F$ satisfies Properties~1 and 2; thus it remains to verify Property~3. We proceed by induction on $|F|$.

    When $|F|=0$, it is trivial that the sequence rises monotonically to $\partial \omega_M(D)$. 
	We assume the conclusion holds for $|F|=k$, and then show it holds for $|F|=k+1$.
		By Lemma \ref{contain}, there exist $F_1\in \mathcal{F}_{\min}(D)$ and $F_2\in \mathcal{F}_{\max}(D)$ such that $$F_1\subseteq F\subseteq F_2.$$ 
		{\bf Case 1.} If $F_1 \neq \emptyset$, let $f \in F_1\subseteq F$. Then
		$$r_{\min}(D\Delta f)=|F_1\Delta f|=r_{\min}(D)-1$$
        and 
		$$r_{\max}(D\Delta f)\geq r_{\max}(D)-1.$$
		Thus $\omega(D\Delta f)\geq \omega(D)$.
        
        Let $D'=D\Delta f$ and 
		$\hat{\mathcal{F'}}$ be the set of feasible sets of $D'$ such that $\omega(D'\Delta A)=\partial\omega_M(D')$ holds for every $A\in \hat{\mathcal{F'}}$. By the definition of $\partial\omega_M(D')$, we have $\partial\omega_M(D')=\partial\omega_M(D)$. 
         We will prove that $F\Delta f$ is a feasible set of minimum
cardinality in $\hat{\mathcal{F}}'$. Since $\omega(D\Delta F) = \partial\omega_M(D)$, by equation (\ref{eqation1}), there exists a feasible set $X$ of $D$ such that $|F\Delta X| = \partial\omega_M(D)$. Then $$|(F\Delta f) \Delta (X\Delta f)| = \partial\omega_M(D) = \partial\omega_M(D'),$$ where $X\Delta f$ is a feasible set of $D'$. Hence, $F\Delta f \in \hat{\mathcal{F}}'$. It remains to show that $F\Delta f$ has minimum cardinality in $\hat{\mathcal{F}}'$.
      Suppose, for the sake of contradiction, that there exists some $T\in \hat{\mathcal{F}}'$ such that $|T| < |F\Delta f|$. Then $T\Delta f\in \hat{\mathcal{F}}$ and $|T\Delta f| < |F|$, which contradicts the fact that $F$ is a feasible set of minimum
cardinality in $\hat{\mathcal{F}}$.
      
      Since $|F\Delta f| = k$, the conclusion follows by the induction hypothesis.\\
		{\bf Case 2.} Assume $F_1=\emptyset$.
		
		{\bf Case 2.1.} For some $F_3\in \mathcal{F}_{\max}(D)$, $F\nsubseteq F_3$. This means that there exists an $f\in F$ and $f\notin F_3$. 
		Then 
		$$r_{\max}(D\Delta f)=|F_3\Delta f|=r_{\max}(D)+1$$
        and 
		$$r_{\min}(D\Delta f)\leq r_{\min}(D)+1.$$
		Thus $\omega(D\Delta f)\geq \omega(D)$. Let $D'=D\Delta f$ and 
		$\hat{\mathcal{F'}}$ be the set of feasible sets of $D'$ such that $\omega(D'\Delta F)=\partial\omega_M(D')$ holds for every $F\in \hat{\mathcal{F'}}$. It is confirmed that $F\Delta f$ is a feasible set of minimum
cardinality in $\hat{\mathcal{F}}'$, and the proof is analogous to that in Case 1. Since $|F\Delta f|=k$, the conclusion thus holds by the induction hypothesis.
		
		{\bf Case 2.2.} For any $A\in \mathcal{F}_{\max}(D)$, $F\subseteq A$.
		Since $\omega(D\Delta F)=\partial\omega_M(D)$,  by equation (\ref{eqation1}), there exists a feasible set $X$ such that $|F\Delta X|=\partial\omega_M(D).$ 
		
		By Lemma \ref{contain}, there is an $F_3\in \mathcal{F}_{\max}(D)$ such that $X\subseteq F_3$. 
		Since $\emptyset=F_1\in \mathcal{F}$, $\omega(D)=r_{\max}-0=|F_3|$, and since $F\neq \emptyset$, we know $\emptyset \notin \hat{\mathcal{F}}$. Then
        $$|F\Delta X|= \partial\omega_M(D)>\omega(D)=|F_3|.$$
		Since $F_3\in\mathcal{F}_{\max}(D)$, $F\subseteq F_3$. Thus, we have $F\cup X \subseteq F_3$, and further
        $$|F_3|\geq |F\cup X|\geq |F\Delta X|>|F_3|,$$ which is impossible. The conclusion holds.
	\end{proof}
    \begin{remark}	
		\normalfont
        The above theorem does not hold for general set systems. For example, let $D = ([5], \{\emptyset, \{1,2\}, \{3,4,5\}, \{1,2,3,4\}\})$. Clearly, $\omega(D)=4$. It is straightforward to verify that only $D\Delta \{1,2\}$ and $D\Delta \{3,4,5\}$ have the maximum twist width of $D$, i.e., $\partial\omega_M(D)=\omega(D\Delta \{1,2\})=\omega(D\Delta \{3,4,5\})=5$. Thus, there is no sequence satisfying Property~(3) of Theorem~\ref{key} that rises monotonically to $\partial\omega_M(D)$.
\end{remark}

    For the delta-matroid $D(G)$ of a ribbon graph $G$, the above theorem immediately implies the following result, which provides an affirmative answer to the Problem \ref{oques} proposed in \cite{Chen2025+}.
	\begin{corollary}
		
		For any ribbon graph $G$, there exists an edge set $A$ (corresponding to a spanning quasi-tree) with a sequence of its edges $e_1,e_2,\dots, e_k$ such that
		$$\gamma(G^ {\{e_1, e_2,\dots, e_k\}})=\gamma_M(G)$$ and the sequence of Euler genera
		$$\gamma(G), \gamma(G^{\{e_1\}}), \dots, \gamma(G^ {\{e_1, e_2,\dots, e_k\}})$$
		rises monotonically (i.e., never decreasing) to $\gamma_M(G)$.
	\end{corollary}
\section{Algorithm and Examples}\label{5}

For a given delta-matroid \( D = (E, \mathcal{F}) \), let \( \hat{\mathcal{F}} \) be the set of feasible sets \( F \) of \( D \) satisfying \( \omega(D \Delta F) = \partial\omega_M(D) \).  Indeed, Theorem \ref{key} provides an algorithm for constructing a sequence that satisfies the monotonic property in Theorem \ref{key}, see Algorithm \ref{al}.
    
    \begin{algorithm}[htbp]
	\caption{Monotonic Sequence Construction for Delta-Matroids}
	\label{alg:delta_matroid_seq_only_S}
	\begin{algorithmic}[1]\label{al}
		\REQUIRE 
		\STATE The delta-matroid $D = (E, \mathcal{F})$;
		\STATE Initial feasible set $F_{\text{init}} \in   \hat{\mathcal{F}}$ such that $|F_{\text{init}}|\leq |T|$ for any $T\in \hat{\mathcal{F}}$;

		\ENSURE 
		\STATE $S = [s_1, s_2, \dots, s_n]$ (where $n = |F_{\text{init}}|$) containing all elements removed from $F_{\text{init}}$ in order.
		
		\STATE \textbf{Initialization}
		\STATE  $S := []$; $F_{\text{curr}} := F_{\text{init}}$; $\mathcal{F}_{\text{curr}} := \mathcal{F}$.

		\STATE \textbf{Iterative Reduction (Core Logic)}
		\WHILE{$F_{\text{curr}}\neq \emptyset$}
		\IF{$\emptyset \notin \mathcal{F}_{\text{curr}}$}
		\STATE  $\emptyset \notin \mathcal{F}_{\text{curr}}$
		\STATE \quad Choose $X \in \mathcal{F}_{\min}(D)$ such that $X \subseteq F_{\text{curr}}$ (exists by Thm. \ref{key}, Case 1);
		\STATE \quad Choose $x \in X$ arbitrarily;
		\STATE \quad $S := S + [x]$;
        \STATE \quad $F_{\text{curr}} := F_{\text{curr}} \setminus \{x\}$;
        \STATE \quad $\mathcal{F}_{\text{curr}} := \{T \Delta \{x\} \mid T \in \mathcal{F}_{\text{curr}}\}$.
        \ELSE
	     \STATE  $\emptyset \in \mathcal{F}_{\text{curr}}$
        \STATE \quad Choose $X \in \mathcal{F}_{\max}(D)$ such that $F_{\text{curr}} \nsubseteq X$ (exists by Thm. \ref{key}, Case 2.1);
        \STATE \quad Choose $x \in F_{\text{curr}} \setminus X$;
        \STATE \quad $S := S + [x]$;
        \STATE \quad $F_{\text{curr}} := F_{\text{curr}} \setminus \{x\}$;
        \STATE \quad $\mathcal{F}_{\text{curr}} := \{T \Delta \{x\} \mid T \in \mathcal{F}_{\text{curr}}\}$.
		\ENDIF
		\ENDWHILE
		
		\STATE \textbf{Finalization}
		\STATE Return the sequence $S$; 
	\end{algorithmic}
\end{algorithm} 
  
   We give a small example of a non-binary delta-matroid due to Bouchet \cite{Bouchet1991}, and construct its monotonic sequence via the proposed algorithm.  
\begin{example}
\normalfont
    Let $D=([4],\mathcal{F})$ where $$\mathcal{F}=\{\emptyset, \{1,2\},\{1,3\},\{1,4\},\{2,3\},\{2,4\},\{3,4\}\}.$$ It is easy to see that $ \hat{\mathcal{F}} =\{\{1,2\},\{1,3\},\{1,4\},\{2,3\},\{2,4\},\{3,4\}\}$. 
    By selecting $F=\{1,2\}$, we start the algorithm.\\
    \begin{enumerate}
        \item \textbf{First iteration}: $\emptyset\in \mathcal{F}$, we select $X=\{2,3\} \in \mathcal{F}_{\max}$;\\
         \quad Choose $1\in F\setminus X$;\\
          \quad Set $S = [1]$, \\
$\qquad F' = F \setminus \{1\} = \{2\}$, \\
$\qquad \mathcal{F}' = \{\{1\}, \{2\},\{3\},\{4\},\{1,2,3\},\{1,2,4\},\{1,3,4\}\}$.\\
          \item \textbf{Second iteration}:
           $\emptyset\notin \mathcal{F}'$, we select $X=\{2\}\in \mathcal{F}'_{\min}$;\\
    \quad Choose $2\in X\subseteq F'$;\\
    \quad Set $S=[1,2]$,\\
    $F'=F\setminus\{2\}=\emptyset$, \\
    $\mathcal{F}'=\{\emptyset, \{1,2\},\{1,3\},\{1,4\},\{2,3\},\{2,4\},\{1,2,3,4\}\}$;\\
    \item \textbf{Termination}: Since $F'=\emptyset$, return $S=[1,2]$. 
    \end{enumerate} 
\end{example}

For a ribbon graph, we can derive its corresponding delta-matroid and then apply the algorithm to obtain a monotonic partial dual Euler genus sequence for the ribbon graph.
\begin{figure}[h]
    \centering
    \includegraphics[width=0.7\linewidth]{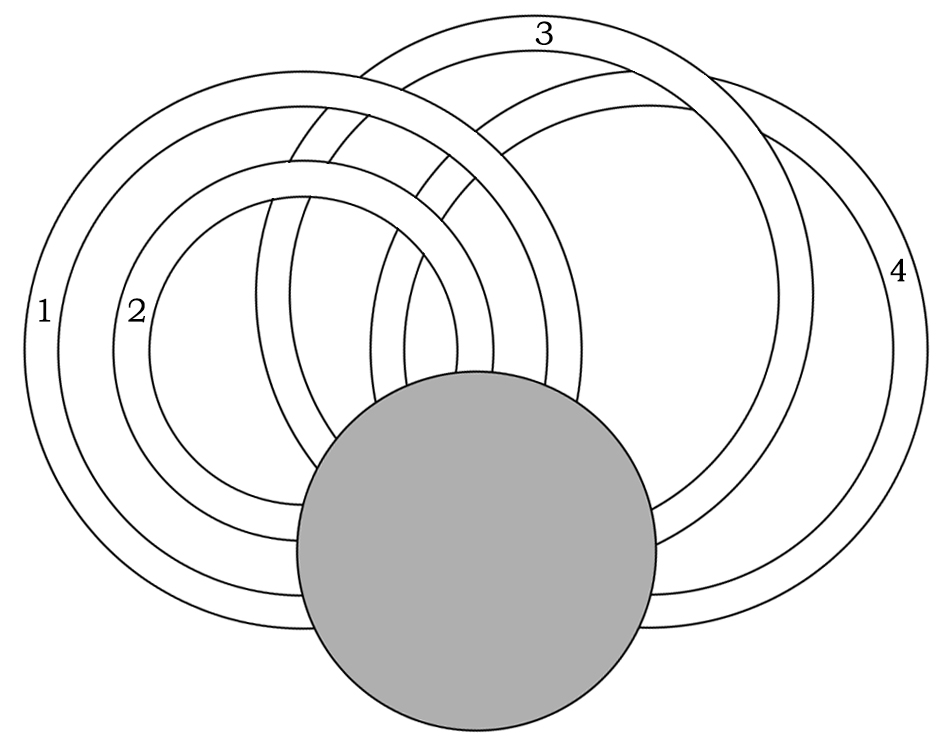}
    \caption{A ribbon graph with one vertex and four edges}
    \label{Fig01}
\end{figure}
\begin{example}
\normalfont
    Let $G$ be a ribbon graph as shown in Figure \ref{Fig01}. Then the delta-matroid of $G$ is $D(G)=([4],\mathcal{F})$, where $\mathcal{F}=\{\emptyset, \{1,3\},\{1,4\},\{2,3\},\{2,4\},\{3,4\}\}$. Clearly, $ \hat{\mathcal{F}} =\{\{1,3\},\{1,4\},\{2,3\},\{2,4\}\}$. 
    By selecting $F=\{1,3\}$, we start the algorithm.\\
    \begin{enumerate}
        \item \textbf{First iteration}: $\emptyset\in \mathcal{F}$, we select $X=\{2,3\} \in \mathcal{F}_{\max}$;\\
         \quad Choose $1\in F\setminus X$;\\
          \quad Set $S=[1]$, $F'=F\setminus\{1\}=\{3\}$, $\mathcal{F}'=\{\{1\},\{3\},\{4\},\{1,2,3\},\{1,2,4\},\{1,3,4\}\}$.\\
          \item \textbf{Second iteration}:
           $\emptyset\notin \mathcal{F}'$, we select $X=\{3\}\in \mathcal{F}'_{\min}$;\\
    \quad Choose $3\in X\subseteq F'$;\\
    \quad Set $S=[1,3]$, $F'=F\setminus\{3\}=\emptyset$, $\mathcal{F}'=\{\emptyset, \{1,2\},\{1,3\},\{1,4\},\{3,4\},\{1,2,3,4\}\}$.\\
    \item \textbf{Termination}: Since $F'=\emptyset$, return $S=[1,3]$. 
    \end{enumerate} 
\end{example}
We observe from the above examples that twisting arbitrary sequence of elements from a feasible set in $\hat{\mathcal{F}}$ produces a width sequence that rises monotonically to the maximum twist width. It is natural to ask whether, for all delta-matroids, twisting arbitrary sequence of elements from a minimum-cardinality feasible set in $\hat{\mathcal{F}}$ can produce a desired width sequence. However, the answer is negative, and we present an example below to interpret this.
\begin{figure}[h]
    \centering
    \includegraphics[width=0.7\linewidth]{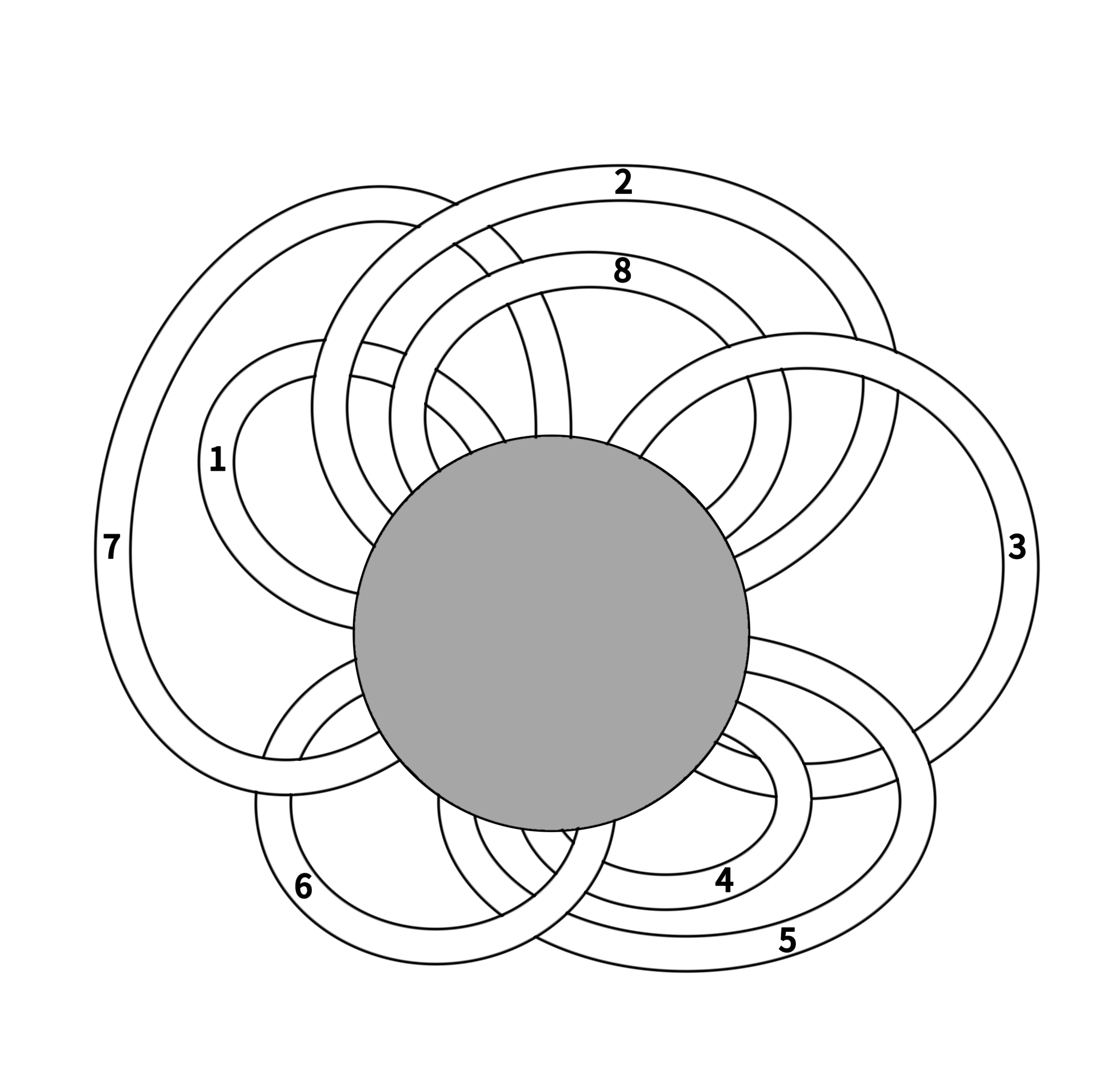}
    \caption{A ribbon graph with one vertex and eight edges}
    \label{Fig02}
\end{figure}
\begin{example}
\normalfont
Let $G$ be a ribbon graph as shown in Figure \ref{Fig02}. Then the delta-matroid of $G$ is $D(G) = ([8], \mathcal{F})$. Let $\mathcal{F}_k$ denote the set of feasible sets of cardinality $k$. It can be routinely checked that 
\begin{align*}
\mathcal{F}_2 &= \{\{1,2\},\{1,8\},\{2,3\},\{2,7\},\{3,4\},\{3,5\},\{3,8\},\{4,6\},\{5,6\},\{6,7\},\{7,8\}\}, \\
\mathcal{F}_6 &= \{\{1,2,3,4,6,7\},\{1,2,3,5,6,7\},\{1,3,4,5,7,8\},\{1,3,5,6,7,8\}\}, 
\end{align*}
and $\{1,2,3,4\},\{5,6,7,8\} \in \mathcal{F}_4$.

For any $F \in \mathcal{F}$, $\omega(D(G) \Delta F) = 8$ if and only if $[8] \setminus F \in \mathcal{F}$. It is easy to verify that for any $F \in \mathcal{F}_2$, $[8]\setminus F \notin \mathcal{F}$ and $\omega(D(G) \Delta F) < 8$. Note that $\omega(D(G) \Delta \{1,2,3,4\}) = 8$ (the maximum twist width of $D(G)$), and thus, $\{1,2,3,4\}$ is a minimum-cardinality feasible set for which $\omega(D(G) \Delta \{1,2,3,4\})$ attains the maximum twist width. However, $\omega(D(G) \Delta \{1\}) = 4 < 6 = \omega(D(G))$. We will see that the width sequence
$$\omega(D),\omega(D(G) \Delta \{1\}),\omega(D(G) \Delta \{1,2\}),\omega(D(G) \Delta \{1,2,3\}),\omega(D(G) \Delta \{1,2,3,4\})$$
does not rise monotonically to $\partial \omega_M(D(G))$.
\end{example}

	\section*{Acknowledgements}
	This work is supported by NSFC (Nos. 12571379, 12471326).

\end{document}